\documentclass [journal,onecolumn,11pt]{IEEEtran}
\usepackage{amsfonts,amsmath,amssymb}

\usepackage{indentfirst, setspace}
\usepackage{url,float}
\usepackage{color}
\usepackage[a4paper,ignoreall]{geometry}
\usepackage{longtable}
\usepackage{graphicx}
\usepackage[a4paper,ignoreall]{geometry}
\usepackage{multicol}
\usepackage{stfloats}
\usepackage{enumerate}
\usepackage{cite}
\usepackage{amsthm}
\usepackage{mathrsfs}
\usepackage{multirow}
\usepackage{amssymb}
\usepackage[all]{xy}
\usepackage[overload]{empheq}
\usepackage[square, comma, sort&compress, numbers]{natbib}
\usepackage{booktabs}
\usepackage{makecell}
\usepackage{tabularx}

\allowdisplaybreaks[4]

\def\qu#1 {\fbox {\footnote {\ }}\ \footnotetext { From Qu: {\color{red}#1}}}
\def\hqu#1 {}
\def\kq#1 {\fbox {\footnote {\ }}\ \footnotetext { From KangQuan: {\color{blue}#1}}}
\def\hkq#1 {}

\newtheorem{Th}{Theorem}[section]

\newtheorem{Prop}[Th]{Proposition}

\newtheorem{Lemma}[Th]{Lemma}
\newtheorem{Def}[Th]{Definition}
\newtheorem{example}{Example}

\newtheorem{Rem}[Th]{Remark}

\newcommand{\tr}{{\rm Tr}}

\newcommand{\gf}{{\mathbb F}}

\newcommand{\PG}{{\rm PG}}

\def\tr{\mathrm{Tr}}

\makeatletter
\newcommand{\figcaption}{\def\@captype{figure}\caption}
\newcommand{\tabcaption}{\def\@captype{table}\caption}
\makeatother

\begin{document}
	\title{Constructing rotatable permutations of $\mathbb{F}_{2^m}^3$ with $3$-homogeneous functions}
	\author{ {Yunwen Chi,  Kangquan Li, Longjiang Qu}
	\thanks{Yunwen Chi,  Kangquan Li, Longjiang Qu are with College of Science,
		National University of Defense Technology, Changsha, 410073, China.
		E-mail: chiyunwen@hotmail.com,  likangquan11@nudt.edu.cn, ljqu\_happy@hotmail.com.
		This work is supported by the National Key Research and Development Program of China under Grant No. 2022YFA1004900, the National Natural Science Foundation of China (NSFC) under Grant 62202476, 62032009 and 62172427,  and the Research Fund of National University of Defense Technology under Grant ZK22-14.
		{\emph{(Corresponding author: Longjiang Qu)}
		}}}
	\maketitle{}

	\begin{abstract}
    In the literature, there are many results about permutation polynomials over finite fields. However, very few permutations of vector spaces are constructed although it has been shown that permutations of vector spaces have many applications in cryptography, especially in constructing permutations with low differential and boomerang uniformities.

    In this paper, motivated by the butterfly structure \cite{perrin2016cryptanalysis} and the work of Qu and Li \cite{qu2023}, we investigate rotatable permutations from $\gf_{2^m}^3$ to itself with $d$-homogenous functions.
    Based on the theory of equations of low degree, the resultant of polynomials,  and some skills of exponential sums, we construct five infinite classes of $3$-homogeneous rotatable permutations from $\gf_{2^m}^3$ to itself, where $m$ is odd. Moreover, we demonstrate that the corresponding permutation polynomials of $\gf_{2^{3m}}$ of our newly constructed permutations of $\gf_{2^m}^3$ are QM-inequivalent to the known ones.

	\end{abstract}


\section{Introduction}
\label{ch1}

Let $\gf_q$ be the field with $q$ elements. A polynomial $F\in \gf_{q}[x]$ is a permutation polynomial (PP for short) of $\gf_q$ if its associate polynomial function $F:\ c \mapsto F(c) $ from $\gf_q$ to itself is a permutation of $\gf_q$.
PPs of finite fields are an important part of finite field theory and have a wide range of applications in cryptography \cite{qu2013constructing,li2021permutation,li2021cryptographically,kim2022permutation}, coding theory \cite{laigle2007permutation,ding2013cyclic},  combinatorial designs \cite{ding2006family}, and many other fields. For example, the key nonlinear component of the AES (Advanced Encryption
Standard) algorithm uses the inverse function, which is a PP over the finite field with 256 elements.
Interested readers can refer to \cite{hou2015permutation,li2019survey,wang2019polynomials} for more comprehensive and recent results about PPs.

Let $n$ be a positive integer and $\gf_q^n$ be the vector space of $\gf_q$ with dimension $n$. We often identify the vector space $\gf_q^n$ with the finite field $\gf_{q^n}$. A function $F$ from $\gf_q^n$ to itself is called a permutation of $\gf_q^n$ if, for all $a\in\gf_q^n$, the equation $F(x)=a$ has exactly one solution in $\gf_q^n$. It is well known that a permutation of $\gf_q^n$ can be uniquely transformed to a permutation polynomial of $\gf_{q^n}$ with a degree less than $q^n-1$, and vice versa. Compared to the results for permutation polynomials, there are not many constructions for permutations of the vector space $\gf_q^n$. But in fact, permutations of $\gf_q^n$ have also applications in many areas. For example, in 2016, regarding the only current example of APN(Almost Perfect Nonlinear) permutation \cite{browning2010apn}, Perrin et al. \cite{perrin2016cryptanalysis} used a reverse engineering method to propose a simple form of the unique APN permutation and generalized it to the butterfly structures, including open and closed ones, which are essentially functions from $\gf_q^2$ to itself.  The closed butterfly structure is a quadratic non-bijective function of $\gf_q^2$ defined by
\begin{equation}
    \label{work1}
    (R(x,y), R(y,x)),
\end{equation}
where $R(x,y)$ is a function from $\gf_q^2$ to $\gf_q$,
and the open one is a permutation, which is CCZ-equivalent \cite{carlet1998codes} to the closed one. Later, Canteaut et al. \cite{canteaut2017generalisation} proposed the generalized butterfly structure and obtained many new permutations with the best-known differential uniformity and nonlinearity by choosing $R(x,y)=(x+\alpha y)^3+\beta y^3$, where $\alpha,\beta\in\gf_q$ with $q=2^m$.  In 2021,  Li et al. \cite{li2021cryptographically} and Li et al. \cite{li20204} obtained a class of permutations on $\gf_{2^m}^2$ with boomerang uniformity 4 (a  newly proposed cryptographic notion for evaluating the subtleties of
boomerang-style attacks,  see \cite{boura2018boomerang,cid2018boomerang,li2019new}) from the open butterfly structure by choosing
\begin{equation}
    \label{work2}
    R(x,y)=(x+\alpha y)^{2^i+1}+\beta y^{2^i+1},
\end{equation}
 where $\alpha,\beta \in \gf_{2^m}$, $\gcd(i,m)=1$. In 2022, Beierle et al. \cite{beierle2022further} gave simple representations of two APN permutation instances on $\gf_{2^9}$ found in \cite{beierle2021new}, namely
\begin{equation}
    \label{work3}
        F(x,y,z)=(x^3+uy^2z,y^3+uz^2x,z^3+ux^2y),
\end{equation}
where $u\in \gf_{2^3}$ is the root of $x^3+x+1$ or $x^3+x^2+1$. Later in 2023, Li and Nikolay \cite{li2022two} obtained the following two infinite classes of APN permutations from $\gf_{2^m}^3$ to itself 
$$
F_1(x,y,z)=(x^3+x^2z+yz^2,x^2z+y^3,xy^2+y^2z+z^3)$$
and
$$F_2(x,y,z)=(x^3+xy^2+yz^2,xy^2+z^3,x^2z+y^3+y^2z),$$ 
    where $m$ is odd. Remarkably, the above families $F_1$ and $F_2$ cover the APN permutation instances \eqref{work3}. From the above results, we can find that constructing permutations of the vector space $\gf_q^n$ may be a new and useful approach to obtaining permutations with good cryptographic properties.

    In addition to directly corresponding to permutation polynomials over finite fields, which is well known, permutations on the vector space $\gf_q^n$ can also be used to construct permutation polynomials of the form $x^rh(x^{q-1})$ over the finite field $\gf_{q^n}$.
    Very recently, inspired by the results of permutation polynomials of the form $x^rh(x^{q-1})$ over $\gf_{q^2}$,
    Qu and Li \cite{qu2023} proposed a bijection from the multiplicative subgroup $\mu_{q^2+q+1}$ of $\gf_{q^3}$ with order $q^2+q+1$ and the projective plane $\mathrm{PG}(2,q)$. They proposed a method to construct permutation polynomials of form $f(x)=x^rh(x^{q-1})$ of $\gf_{q^3}$ via  bijections of $\PG(2,q)$, mainly including permutations of $\gf_q^3$ with $d$-homogeneous monomials and $2$-homogeneous functions.

It is not difficult to find that all the existing results of permutations of vector spaces have one thing in common, that is, they are all $d$-\textit{homogeneous}. For a given positive integer $d$, a polynomial $f(x_1,x_2,\ldots,x_n)\in\gf_q[x_1,x_2,\ldots,x_n]$ is called $d$-homogeneous if for all $\lambda\in\gf_q$, $$f(\lambda x_1, \lambda x_2,\ldots, \lambda x_n) = \lambda^d f(x_1,x_2,\ldots,x_n). $$ The well-known $d$-homogeneous functions have many applications in sequence theory \cite{klapper1995d}, difference sets \cite{kim2005new,no2004new}, etc. For example, permutations with boomerang uniformity $4$ obtained by  Li et al. \cite{li2021cryptographically} and Li et al. \cite{li20204} are $(2^i+1)$-homogeneous; APN permutations obtained by Li and Nikolay \cite{li2022two}  are $3$-homogeneous.  In addition, the closed butterfly structure \eqref{work1} and the two APN permutation instances \eqref{work3} are also rotatable. For a function $F = (f_1,f_2,\ldots,f_n)$ from $\gf_q^n$ to itself, we say that $F$ is \textit{rotatable} if for all $1\le i\le n$, $f_i(x_1,x_2,\ldots,x_n)=f_1(\sigma_{i-1}(x_1,x_2,\ldots,x_n))$, where $$\sigma_i= \underbrace{\sigma_1\circ\sigma_1 \circ \cdots \circ\sigma_1}_{i~\text{times}},$$ $\sigma_0$ is the  identity mapping and  $\sigma_1(x_1,x_2,\ldots,x_n)=(x_2,x_3,\ldots,x_n,x_1)$.

Motivated by the above observations, in this paper, we study rotatable permutations from  $\gf_q^3$ to itself with $3$-homogeneous functions. Specifically, we consider permutations on $\gf_q^3$ with $q=2^m$( $m$ is odd) of the form
\begin{equation}
\label{Fxyz}
   F(x,y,z)=(f(x,y,z),f(y,z,x),f(z,x,y)),
\end{equation}
where
$$
f(x,y,z)= x^3+a_1y^3+a_2z^3+a_3x^2y+a_4xy^2+a_5x^2z+a_6xz^2+a_7yz^2+a_8y^2z
$$
and $a_i\in\gf_2$ for $i\in \{1,2,...,8\}$. After investigating some basic properties of $d$-homogeneous permutations.
we totally construct five infinite classes of rotatable permutations from $\gf_{q}^3$ to itself where $q=2^m$, $m$ is odd. The main difficulty in the proofs lies in solving multivariate systems of equations. Two different methods are used to prove the results above, which share a common idea: elimination. The first method takes advantage of the special structure of rotation. The second method employs a combination of resultant and exponential sums to transform the equation system into a single one, and further analyses its solution.  Our permutations of $\gf_q^3$ have their polynomial representations of the finite field $\gf_{q^3}$ due to the isomorphism between $\gf_{q^3}$ and $\gf
_{q}^3$. Recently, many researchers also studied permutation trinomials of $\gf_{q^3},$ see \cite{xie2023two,bartoli2020permutation,gupta2022new,DBLP:journals/ffa/ZhaHZ19,wang2018six,bartoli2023permutation}. Since the corresponding permutation polynomials of $\gf_{q^3}$ of our newly constructed permutations on $\gf_q^3$ have much more than three terms, our results are obviously QM-inequivalent to the known ones.

The rest of the paper is organized as follows. In Section \ref{ch2}, we systematically introduce the concepts used in this paper and the lemmas that contribute to the subsequent proofs.
Five infinite classes of $3$-homogeneous rotatable permutations are presented in Section \ref{ch3}. Section \ref{ch4} concludes the work of this paper and provides some further work.

\section{Preliminaries}
\label{ch2}
We first unify the notation in this paper.

\begin{itemize}
    \item $\gf_{q}$: the finite field with $q$ elements.
    \item $\gf_{q}^n$: $n$-dimension vector space of $\gf_{q}$.
    \item $\tr_{q^n/q}(\alpha)$: the trace function of $\alpha$ from $\gf_{q^n}$ to $\gf_{q}$, i.e.,
    $$
    \tr_{q^n/q}(\alpha)=\alpha+\alpha^{q}+\cdots+\alpha^{q^{n-1}}.
    $$
    In particular, we denote it as $\tr_n(\alpha)$ when $q = 2$.
\end{itemize}

Lemma  \ref{lemma-1} shows some properties of the trace function.
\begin{Lemma}
\cite{lidl1997finite,hou2015permutation,li2023further}
\label{lemma-1}
    Let $K=\gf_q$ and $F=\gf_{q^m}$. Then the trace function $\tr_{F/K}$ satisfies the following properties:

    (i) $\tr_{F/K}(\alpha+\beta)=\tr_{F/K}(\alpha)+\tr_{F/K}(\beta)$ for all $\alpha, \beta \in F$;

    (ii) $\tr_{F/K}(c\alpha)=c \tr_{F/K}(\alpha)$ for all $c\in K$, $\alpha\in F$;

    (iii) $\tr_{F/K}(\alpha^q)= \tr_{F/K}(\alpha)$ for all $\alpha\in F$;

    (iv) $\sum\limits_{\omega\in \mathbb{F}_{2^m}}(-1)^{\mathrm{Tr}_m(\alpha\omega)}=\left\{
            \begin{aligned}
                &\ 0,\ \alpha\neq 0,\\
                &\ 2^m,\ \alpha=0.
            \end{aligned}
            \right.$
\end{Lemma}

Lemmas \ref{lemma-4} and \ref{lemma-3} characterize the number of solutions of quadratic and cubic equations respectively.

\begin{Lemma}
    \cite{lidl1997finite,williams1975note}
    \label{lemma-4}
        Let $a,b\in \mathbb{F}_{2^n}$ and $a\neq 0$. Then the quadratic equation $x^2+ax+b=0$ has no solution in $\mathbb{F}_{2^n}$ if and only if $\mathrm{Tr}_n(a^{-2}b)\neq 0$.
    \end{Lemma}

 \begin{Lemma}
    \cite{berlekamp1967solution}
    \label{lemma-3}
        Let $a,b\in \mathbb{F}_{2^n}$ and $b\neq 0$. Then the cubic equation $x^3+ax+b=0$ has a unique solution in $\mathbb{F}_{2^n}$ if and only if $\mathrm{Tr}_n\left(\frac{a^3}{b^2}+1\right)\neq 0$.
    \end{Lemma}

Further, we will define the resultant of two polynomials, which is useful in our proof.

\begin{Def}
    \cite{lidl1997finite}
    \label{def-1}
    Let ${K}$ be a field, $f(x)=a_0x^n+a_1x^{n-1}+\cdots+a_n \in {K}[x]$ and $g(x)=b_0 x^m+b_1 x^{m-1}+ \cdots +b_m \in {K}[x]$ be two polynomials of formal degree $n$ resp. $m$ with $n,m \in \mathbb{N}$. Then the resultant $R(f,g)$ of the two polynomials is defined by the determinant
    $$
    R(f,g)=\left |
    \begin{matrix}
        a_0 & a_1 & \cdots & a_n & 0   &   & \cdots & 0\\
        0   & a_0 & a_1 & \cdots & a_n & 0 & \cdots & 0\\
        \vdots&   &     &     &     &   &     & \vdots \\
        0   & \cdots & 0   & a_0 & a_1 &   & \cdots & a_n \\
        b_0 & b_1 & \cdots &     & b_m & 0 & \cdots & 0 \\
        0   & b_0 & b_1    & \cdots&   & b_m& \cdots& 0 \\
        \vdots&   &     &     &     &   &     & \vdots \\
        0& \cdots & 0   & b_0 & b_1 &   & \cdots & b_m \\
    \end{matrix}
    \right|
    $$
    of order $m+n$.
\end{Def}

For a field $K$ and two polynomials $F(x,y), G(x,y) \in K[x,y]$, we use $ R_y(F,G)$ to denote the resultant of $F$ and $G$ with respect to $y$. It is the resultant of $F$ and $G$ when considered as polynomials in the single variable $y$. In this case, $ R_y(F,G)\in K[x]$ belongs to the ideal generated by $F$ and $G$, and thus any $a,b$ satisfying $F(a,b)=0$ and $G(a,b)=0$ is such that $R_y(F,G)(a)=0$ (see \cite{lidl1997finite}).


While proving the permutability of a function from $\mathbb{F}_q^n$ to $\mathbb{F}_q^n$, Lemma \ref{lemma-2} can be a powerful tool to help simplify the structure of the function.

 \begin{Lemma}
    \label{lemma-2}
        Let $\phi$,$\psi$ be permutations from $\mathbb{F}_q^n$ to $\mathbb{F}_q^n$, and $F$ is a function from $\mathbb{F}_q^n$ to $\mathbb{F}_q^n$. Then $\psi \circ F \circ \phi$ is a permutation if and only if $F$ is a permutation.
\end{Lemma}

Let $F$ be a function from $\gf_q^3$ to itself. If $F$ is a permutation, the following lemma gives two sufficient and necessary conditions respectively.

\begin{Lemma}
\cite{lidl1997finite}
    \label{lemma-5}
    Let $F(x,y,z)$ be a function from $\gf_q^3$ to itself. Then $F$ is a permutation if and only if any of the following conditions hold:

    (i) the equation $F(x,y,z)=(a,b,c)$ has only one solution for all $a,b,c \in \gf_q$;

    (ii) the equation $F(x,y,z)=F(x+a,y+b,z+c)$ has no solution for all $(a,b,c)\in \gf_q^3\backslash (0,0,0)$.
\end{Lemma}

In this paper, we consider permutations of the form
$$F(x,y,z) = (f(x,y,z),f(y,z,x),f(z,x,y)),$$
where $f(x,y,z)$ is $3$-homogeneous. For the $d$-homogeneous rotatable permutations, we have the following necessary conditions.

\begin{Prop}
    Let $d$ be a positive integer and $F(x,y,z) = (f(x,y,z),f(y,z,x),f(z,x,y)),$ where $f(x,y,z)\in\gf_q[x,y,z]$ is $d$-homogeneous. If $\gcd(d,q-1)\neq1$, then $F$ is not a permutation of $\gf_q^3$.
\end{Prop}
\begin{proof}
If  $\gcd(d,q-1)\neq1$, then there exist some elements $\lambda\in\gf_q\backslash\{1\}$ such that $\lambda^d=1$. Then we have
$$F(\lambda x, \lambda y, \lambda z) = \lambda^d (f(x,y,z),f(y,z,x),f(z,x,y)) = F(x,y,z)$$
for all $x,y,z\in\gf_q$, while $(\lambda x, \lambda y, \lambda z)\neq (x,y,z)$. Thus $F$ is not a permutation of $\gf_q^3.$
\end{proof}

\begin{Rem}
    In this paper, we mainly consider permutations $F$ of $\gf_{2^m}^3$ from $3$-homogeneous functions. Thus we always have that $m$ is odd since $\gcd(3,2^m-1)=3$ when $m$ is even and then $F$ is not a permutation.
\end{Rem}

When a permutation is constructed, an important thing is to show that it is inequivalent to known results. The following is a frequent equivalent relation among permutations.

\begin{Def}
\cite{wu2017permutation}
    Two permutation polynomials $F(x)$ and $G(x)$ in $\gf_q[x]$ are called quasi-multiplicative (QM for short) equivalence if there exists an integer $1\le d\le q-1$ with $\gcd(d,q-1)=1$ and $F(x)=aG(cx^d)$, where $a,c\in\gf_q^{*}$.
\end{Def}
\begin{Rem}
\label{rem}
    It is obvious that if $F$ and $G$ have different numbers of terms, $F$ and $G$ must be QM-inequivalent.
\end{Rem}

\section{Five infinite classes of $3$-homogeneous rotatable permutations of $\gf_{2^m}^3$}
\label{ch3}

In this section, we construct five infinite classes of $3$-homogeneous rotable permutations of $\gf_{2^m}^3$. Our proofs depend on the lemmas in the above section, the resultant of polynomials, and some skills of exponential sums.



\begin{Th}
    \label{th-4}
    Let $m$ be odd and
    $F(x,y,z)=(f(x,y,z),f(y,z,x),f(z,x,y))$ be a function from $\gf_{2^m}^3$ to itself, where $$f(x,y,z)=x^3+y^3+x^2z+xy^2+yz^2.$$
 Then $F$ is a permutation of $\gf_{2^m}^3$.
\end{Th}
\begin{proof}
    Let $\phi(x,y,z)=(y+z,x+z,x+y+z)$, which is obviously a permutation from $\mathbb{F}_{2^m}^3$ to $\mathbb{F}_{2^m}^3$. Then $F$ is a permutation if and only if so is $\phi\circ F,$ which is
    $$\phi\circ F(x,y,z)=(x^3+y^3,y^3+z^3,xy^2+yz^2+x^2z).$$

    Let $H=\phi \circ F$. In the following, we prove that $H$ is a permutation.
    According to the definition, it suffices to show that for all $a,b,c\in\gf_{2^m}$, the equations system
    \begin{subequations}
\renewcommand\theequation{\theparentequation.\arabic{equation}}
	      \label{iv-eq1}
	   \begin{empheq}[left={\empheqlbrace\,}]{align}
        &x^3+y^3=a \label{2.1}\\
        &y^3+z^3=b \label{2.2}\\
        &xy^2+yz^2+x^2z=c \label{2.3}
        \end{empheq}
    \end{subequations}
    has unique solution on $\mathbb{F}_{2^m}^3$.


Let $P_1=x^3+y^3+a, P_2 = y^3+z^3+b$ and $P_3=xy^2+yz^2+x^2z+c$. Then $P_1, P_2, P_3$ can be seen as polynomials in $\gf_2[x,y,z,a,b,c]$.
  In order to eliminate the variable $x$, we first compute the resultant of $P_1$ and $P_3$ aiming at $x$ by MAGMA and get
    \begin{equation}
    \label{2.4}
    \begin{aligned}
        g &= R_x(P_1, P_3) \\
        &=  y^9 + y^6a + y^5zc + y^3z^6 + y^3z^3a + y^2z^4c \\
        &\ \ +y^2zac + yz^2c^2 + z^3a^2 + c^3.
    \end{aligned}
    \end{equation}

    Next, we can obtain the resultant of $g$ and $P_2$ aiming at $z$ by MAGMA, that is
    \begin{equation}
    \label{2.5}
    \begin{aligned}
        h&=R_z(g, P_2)\\
        &=(a^6 + a^5b + a^3b^3 +a^2bc^3 + ab^5 + ab^2c^3 + b^6 + c^6)y^9\\
        &+(a^6 b +a^4 b^3 +a^4 c^3 +a^3 b c^3 +a^2 b^5 +a^2 b^2 c^3 +a b^3 c^3 +a c^6 +b^4 c^3+b c^6)y^6\\
        &+(a^6 b^2 +a^5 b^3 +a^4 b^4 +a^3 b^2 c^3 +a^2 b^3 c^3 +a^2 c^6 +b^2 c^6)y^3\\
        &+(a^6 b^3 + a^4 b^2 c^3 + a^2 b c^6 + c^9).
    \end{aligned}
    \end{equation}
    According to the resultant theory, if $(x_0,y_0,z_0)$ is a solution of the equation system \eqref{iv-eq1}, then $y_0$ is a root of $h$. Thus we consider $h=0$.

    Assume $Y=y^3$, Then the equation $h=0$ becomes
    \begin{equation}
    \label{2-1}
        AY^3+BY^2+CY+D=0
    \end{equation}
    where
    $$
    \begin{cases}
        A=a^6 + a^5b + a^3b^3 +a^2bc^3 + ab^5 + ab^2c^3 + b^6 + c^6,\\
        B=a^6 b +a^4 b^3 +a^4 c^3 +a^3 b c^3 +a^2 b^5 +a^2 b^2 c^3 +a b^3 c^3 +a c^6 +b^4 c^3+b c^6,\\
        C=a^6 b^2 +a^5 b^3 +a^4 b^4 +a^3 b^2 c^3 +a^2 b^3 c^3 +a^2 c^6 +b^2 c^6,\\
        D=a^6 b^3 + a^4 b^2 c^3 + a^2 b c^6 + c^9.
    \end{cases}
    $$
    Note that $A = (a^2 + a b + b^2 + b c + c^2)(a^2 + a b + a c + b^2 + c^2)(a^2 + a b + a c + b^2 + b c + c^2)$.
Now we consider the condition of $A=0$, which implies $a^2 + a b + b^2 + b c + c^2=0$, or $a^2 + a b + a c + b^2 + c^2=0$, or $a^2 + a b + a c + b^2 + b c + c^2=0$. If $a^2 + a b + b^2 + b c + c^2=(a+c)^2+(a+c)b+b^2=0$ holds, we have
$$\frac{a^2}{b^2}+\frac{a}{b}+1+\frac{c}{b}+\frac{c^2}{b^2}=0.$$
when $b\neq 0$, and the trace of the left side of the equation above is
$$\tr_m\left(\frac{a^2}{b^2}+\frac{a}{b}+1+\frac{c}{b}+\frac{c^2}{b^2}\right)=\tr_m(1)\neq \tr_m(0),$$
which is a contradiction! Therefore, we have $b=0$ and $a=c$. Similarly, $a^2 + a b + a c + b^2 + c^2=0$ implies $a=0$ and $b=c$. And $a^2 + a b + a c + b^2 + b c + c^2=(a+c)(a+b)+(a+b)^2+(a+c)^2=0$ implies $a+b=0$ and $a+c=0$, i.e. $a=b=c$.

Therefore, $A=0$ if and only if any one of the following three conditions is satisfied: 

(i) $b=0$ and $a=c$; 

(ii) $a=0$ and $b=c$; 

(iii) $a=b=c$, 

We first consider the most special case $a=b=c=0$. In this case, by Eq. \eqref{2.1} and Eq. \eqref{2.2}, we have $x=y=z$ since $\gcd(3,2^m-1)=1$. Moreover, by Eq. \eqref{2.3}, we have $x^3=0$, and thus $(0,0,0)$ is the unique solution of the equation system \eqref{iv-eq1}. Next, we consider the case $b=0$ and $a=c\neq0$. In this case, by Eq. \eqref{2.2}, we have $y=z$. Plugging it and $a=c$ into Eq. \eqref{2.3}, we get $xy^2+y^3+x^2y=a$, i.e., $(x+y)^3+x^3=a$. Since  $\gcd(2^m-1,3)=1$, there exists the inverse of $3$ module $2^m-1$, denoted by $1/3$. Together with Eq. \eqref{2.1},  we have $(x+y)^3=y^3$ which implies $x=0, y=a^{1/3}$. So $(0,a^{1/3},a^{1/3})$ is the unique solution of the equation system \eqref{iv-eq1}. For the case $a=0$ and $b=c\neq0$, we can also show that the equation system \eqref{iv-eq1} has a unique solution. Since the discussion is similar to that of the above case, we omit it here. Finally, we consider the case $a=b=c\neq0$. By Eq. \eqref{2.1} and Eq. \eqref{2.2}, we know $x^3=z^3$, i.e., $x=z$. Plugging it and $a=b=c$ into Eq. \eqref{2.3}, we have $xy^2+x^2Y+x^3=(x+y)^3+y^3=a$. Together with Eq. \eqref{2.1}, we get $(x+y)^3=x^3$ which implies $y=0$ and $x=a^{1/3}$. Therefore, $(a^{1/3},0,a^{1/3})$ is the unique solution of the equation system \eqref{iv-eq1}. Thus when $A=0$, the equation system \eqref{iv-eq1} has exactly one solution.

In the following, we consider that $A\neq0$. Let $Y=Y_1+A^{-1}B$, then Eq. \eqref{2-1} becomes
     \begin{equation}
     \label{2.6}
         Y_1^3+A^{-2}(AC+B^2)Y_1+A^{-2}(AD+BC)=0.
     \end{equation}
Note that $AD+BC=c^3(ab^2+c^3)(a^2b+b^3+c^3)(a^2b+ab^2+c^3)(a^3+a^2b+c^3)$.
If  $AD+BC=0$, it implies that at least one of the following five conditions holds:

(i) $c=0$; 

(ii) $ab^2+c^3=0$; 

(iii) $a^2b+b^3+c^3=0$;
 
(iv) $a^2b+ab^2+c^3=0$; 

(v)  $a^3+a^2b+c^3=0$. 

Moreover, Eq. \eqref{2-1} becomes
$$
Y_1^3+A^{-2}(AC+B^2)Y_1=0.
$$
If $AC+B^2=0$, the above equation has exactly only one solution. If $AC+B^2\neq 0$,
obviously, $Y_1=0$ and $Y_1=A^{-1}(AC+B^2)^{1/2}$ are the solutions. Due to $ Y_1 =y^3+A^{-1}B$, we get two possible solutions of $y^3$: $A^{-1}B$ or $(A^{-1}C)^{1/2}$. We claim that only one of them satisfies the equation system  \eqref{iv-eq1}.

Note that $AC+B^2=abc^3(a+b)(a^2 + a b + b^2)(a^2 b + a b^2 + c^3)(a^3 + a b^2 + b^3 + c^3)\neq 0$, which implies that $a,b,c \neq 0$ and $a+b\neq 0$. Now we discuss the five conditions of $AD+BC=0$ when $AC+B^2\neq 0$. Firstly, when condition (i) or (iv) holds, we get $AC+B^2=0$, which makes Eq. \eqref{2-1} have only one solution. Next, we discuss the remaining conditions. When condition (ii) holds, plugging $ab^2+c^3=0$ into $A, B, C, D$, we get
$$
\begin{cases}
    A=(a+b)(a^5+b^5),\\
    B=ab(a+b)^5,\\
    C=a^2b^2(a^4+a^3b+a^2b^2+ab^3+b^4),\\
    D=a^3b^3(a+b)^3.
\end{cases}
$$
Then $A^{-1}B=\frac{ab(a+b)^4}{a^5+b^5}$ and $(A^{-1}C)^{1/2}=\frac{ab}{a+b}$. Note that $A^{-1}B\neq (A^{-1}C)^{1/2}$ since $a+b\neq 0$ and $a,b\neq 0$. Plugging $y^3=\frac{ab(a+b)^4}{a^5+b^5}$ into     Eq. \eqref{2.1} and Eq. \eqref{2.2} resp. , we have $x^3=\frac{a^5(a+b)}{a^5+b^5}$ and $z^3=\frac{b^5(a+b)}{a^5+b^5}$. Then plugging them into Eq. \eqref{2.3}, the left side is
$$
\begin{aligned}
    &\frac{(a^7b^2(a+b)^9)^{1/3}+(ab^{11}(a+b)^6)^{1/3}+(a^{10}b^5(a+b)^3)^{1/3}}{a^5+b^5}\\
    &=\frac{a^2(a+b)^3+b^3(a+b)^2+a^3b(a+b)}{a^5+b^5}(ab^2)^{1/3}=c,
\end{aligned}
$$
which satisfies the equation system  \eqref{iv-eq1}.

Further, plugging $y^3=\frac{ab}{a+b}$ into Eq. \eqref{2.1} and Eq. \eqref{2.2}, we have $x^3=\frac{a^2}{a+b}$ and $z^3=\frac{b^2}{a+b}$. Finally, plugging them into Eq. \eqref{2.3}, the left side of the equation is
$$
\begin{aligned}
    &\frac{(a^4b^2)^{1/3}+(ab^5)^{1/3}+(a^4b^2)^{1/3}}{a+b}\\
    &=\frac{b}{a+b}(ab^2)^{1/3}\\
    &=\frac{b}{a+b}c\neq c,
\end{aligned}
$$
which is a contradiction. Therefore, Eq. \eqref{iv-eq1} has only one solution
$$\left( \left(\frac{a^5(a+b)}{a^5+b^5}\right)^{1/3},\left(\frac{ab(a+b)^4}{a^5+b^5}\right)^{1/3}, \left(\frac{b^5(a+b)}{a^5+b^5}\right)^{1/3} \right).$$

When the condition (iii) holds, similarly plugging $a^2b+b^3=c^3$ into $A, B, C, D$, we get
$$
\begin{cases}
    A=a^2(a^4+a^3b+b^4),\\
    B=a^2b^5,\\
    C=b^4(a^4+a^3b+b^4),\\
    D=b^9.
\end{cases}
$$
Thus $y^3=\frac{b^5}{a^4+a^3b+b^4}$ or $\frac{b^2}{a}$.

On the one hand, we get the solution $x=(\frac{(a+b)^5}{a^4+a^3b+b^4})^{1/3},y=(\frac{b^5}{a^4+a^3b+b^4})^{1/3},z=(\frac{a^3b(a+b)}{a^4+a^3b+b^4})^{1/3}$ from  $y^3=\frac{b^5}{a^4+a^3b+b^4}$, Eq. \eqref{2.1} and Eq. \eqref{2.2}. After plugging them into Eq. \eqref{2.3}, we have
$$
\begin{aligned}
    &\frac{(b^{10}(a+b)^5)^{1/3}+( a^6b^7(a+b)^2)^{1/3}+(a^3b(a+b)^{11})^{1/3}}{a^4+a^3b+b^4}\\
    &=\frac{b^3(a+b)+a^2b^2+a(a+b)^3}{a^4+a^3b+b^4}(b(a+b)^2)^{1/3}=c,
\end{aligned}
$$
which satisfies the equation system  \eqref{iv-eq1}.

On the other hand, if $y^3=\frac{b^2}{a}$, we compute the solution of $x,y,z$ and plug it into Eq. \eqref{2.3}, then
$$
\begin{aligned}
    &\frac{(b^4(a+b)^2)^{1/3}+(b^4(a+b)^2)^{1/3}+(b(a+b)^5)^{1/3}}{a}\\
    &=\frac{a+b}{a}c\neq c,
\end{aligned}
$$
which is a contradiction. And for condition (v) which is similar to (iii), we omit the proof here. Therefore when $AB+BC=0$, the equation system \eqref{iv-eq1} has exactly one solution. 

Finally, considering the case $AD+BC\neq 0$, we have
    \begin{equation*}
    \label{2.7}
        \mathrm{Tr}_m\left(\frac{[A^{-2}(AC+B^2)]^3}{[A^{-2}(AD+BC)]^2}\right)=\tr_m\left(\frac{(AC+B^2)^3}{A^2(AD+BC)^2}\right).
    \end{equation*}
    If there exists $\alpha \in \mathbb{F}_{2^m}$ such that
    \begin{equation}
    \label{2.8}
        \alpha^2+\alpha= \frac{(AC+B^2)^3}{A^2(AD+BC)^2},
    \end{equation}
    then  $\mathrm{Tr}_m\left(\frac{(AC+B^2)^3}{A^2(AD+BC)^2}\right)=0,$ and by Lemma \ref{lemma-3}, Eq. \eqref{2.6} has only one solution and thus so does Eq. \eqref{2-1}.

    In the following, we will find the element $\alpha\in\gf_{2^m}$ satisfying Eq. (\ref{2.8}).

    Let $\alpha=\frac{\beta}{A(AD+BC)}$. Then Eq. (\ref{2.8}) is equivalent to
    \begin{equation}
    \label{2.9}
        \beta^2+A(AD+BC)\beta=(AC+B^2)^3.
    \end{equation}
    Let $K_1=(AC+B^2)^3$, $K_2=A(AD+BC)$, i.e.
    \begin{align*}
        K_1=&a^{27} b^6 c^9 + a^{26} b^7 c^9 + a^{25} b^5 c^{12 }+ a^{24} b^6 c^{12} + a^{23} b^7 c^{12} + a^{23} b^4 c^{15} + a^{22} b^{11} c^9 +\\
    &a^{22} b^8 c^{12} + a^{21} b^3 c^{18} + a^{20} b^{10} c^{12} + a^{20} b^7 c^{15} + a^{20} b^4 c^{18} + a^{19} b^{14} c^9 + a^{19} b^{11} c^{12} +\\
    &a^{19} b^5 c^{18} + a^{18} b^{15} c^{9} + a^{18} b^{12} c^{12} + a^{18} b^9 c^{15} + a^{18} b^6 c^{18} + a^{18} b^3 c^{21} + a^{17} b^{16} c^9 +\\
    &a^{17} b^{10} c^{15} + a^{17} b^7 c^{18} + a^{17} b^4 c^{21} + a^{16} b^{17} c^9 + a^{16} b^{11} c^{15} + a^{15} b^{18} c^9 + a^{15} b^{15} c^{12} +\\
    &a^{15} b^3 c^{24} + a^{14} b^{16} c^{12} + a^{14} b^7 c^{21} + a^{14} b^4 c^{24} + a^{13} b^{20} c^9 + a^{13} b^{17} c^{12} + a^{13} b^8 c^{21} +\\
    &a^{12} b^{12} c^{18} + a^{12} b^9 c^{21} + a^{12} b^3 c^{27} + a^{11} b^{22} c^9 + a^{11} b^{16} c^{15} + a^{11} b^{10} c^{21} + a^{11} b^7 c^{24} +\\
    &a^{10} b^{20} c^{12} + a^{10} b^{11} c^{21} + a^9 b^{21} c^{12} + a^9 b^{12} c^{21} + a^9 b^6 c^{27} + a^8 b^{25} c^9 + a^8 b^{19} c^{15} +\\
    &a^8 b^{16} c^{18} + a^8 b^{13} c^{21} + a^8 b^{10} c^{24} + a^7 b^{23} c^{12} + a^7 b^{17} c^{18} + a^7 b^{14} c^{21} + a^6 b^{27} c^9 +\\
    \end{align*}
    \begin{align*}
    &a^6 b^{21} c^{15} + a^6 b^9 c^{27} + a^5 b^{25} c^{12} + a^5 b^{22} c^{15} + a^5 b^{19} c^{18} + a^5 b^{13} c^{24} + a^4 b^{23} c^{15} +\\
    &a^4 b^{17} c^{21} + a^3 b^{21} c^{18} + a^3 b^{18} c^{21} + a^3 b^{15} c^{24} + a^3 b^{12} c^{27},
    \end{align*}

    $$
    \begin{aligned}
        K_2=&a^{14} b^4 c^3 + a^{13} b^5 c^3 + a^{13} b^2 c^6 + a^{12} b^3 c^6 + a^{11} b^7 c^3 + a^{10} b^8 c^3 + a^{10} b^5 c^6 + a^{10} b^2 c^9 +\\
    &a^9 b^6 c^6 + a^9 c^{12} + a^8 b^{10} c^3 + a^8 b^7 c^6 + a^8 b^4 c^9 + a^7 b^{11} c^3 + a^7 b^8 c^6 + a^6 b^6 c^9 +\\
    &a^6 b^3 c^{12} + a^6 c^{15} + a^5 b^{13} c^3 + a^5 b^{10} c^6 + a^5 b^4 c^{12} + a^4 b^{14} c^3 + a^4 b^5 c^{12} + a^4 b^2 c^{15} +\\
    &a^3 b^6 c^{12} + a^3 c^{18} + a^2 b^{13} c^6 + a^2 b^{10} c^9 + a b^8 c^{12} + a b^2 c^{18} + b^9 c^{12} + b^6 c^{15} + b^3 c^{18} +\\
    &c^{21}.
    \end{aligned}
    $$

    Next, consider $K_1$, $K_2$ and $\beta$ as functions with respect to $c$. The degree of $K_1$ is 27 while the degree of $K_2$ is 21. Therefore, we can assume that $\beta$ has the form below.

    \begin{equation}
        \label{2.10}
    \beta=p_6c^6+p_5c^5+p_4c^4+p_3c^3+p_2c^2+p_1c+p_0.
    \end{equation}
    And we can also change $K_1,\ K_2$ into the form like equation (\ref{2.10}).
    $$
    K_1=t_1 c^{27}+ t_2 c^{24}+t_3 c^{21}+t_4 c^{18}+ t_5 c^{15} +t_6 c^{12}+t_7 c^9;
    $$
    $$
    K_2=s_1 c^{21}+s_2 c^{18}+ s_3 c^{15}+ s_4 c^{12}+ s_5 c^{9}+s_6 c^6+ s_7 c^3,
    $$
in which $t_i,\ s_i \in \mathbb{F}_2[a,b]$, for $i=1,2,...,7$.

    Then the equation (\ref{2.9}) is equivalent to
    $$
    \begin{aligned}
        \beta^2+K_2\beta+K_1=&(p_6 s_1+t_1)c^{27}  + p_5 s_1c^{26}  + p_4 s_1 c^{25} + (p_6 s_2 + p_3 s_1+t_2)c^{24} \\
        &+( p_5 s_2 +p_2 s_1)c^{23} +  (p_4 s_2 + p_1 s_1)c^{22} +(p_6 s_3 + p_3 s_2 + p_0 s_1 + t_3)c^{21}\\
        &+(p_5 s_3 + p_2 s_2) c^{20} + (p_4 s_3 +p_1 s_2)c^{19} + (p_6 s_4 + p_3 s_3 + p_0 s_2+ t_4)c^{18} \\
        &+(p_5 s_4 + p_2 s_3)c^{17} +(p_4 s_4 + p_1 s_3)c^{16} + (p_6 s_5 + p_3 s_4 + p_0 s_3 + t_5)c^{15} \\
        &+(p_5 s_5 + p_2 s_4)c^{14}  + (p_4 s_5 + p_1 s_4)c^{13}  + (p_6^2 + p_6 s_6 + p_3 s_5 + p_0 s_4+ t_6)c^{12}  \\
        &+(p_5 s_6 + p_2 s_5)c^{11} +(p_5^2 + p_4 s_6 + p_1 s_5)c^{10} +(p_6 s_7 + p_3 s_6 + p_0 s_5 + t_7)c^9 \\
        &+(p_5 s_7 + p_4^2 + p_2 s_6)c^8 +(p_4 s_7 + p_1 s_6)c^7  + (p_3^2 + p_3 s_7 + p_0 s_6)c^6 \\
        &+(p_2 s_7 c^5  + (p_2^2 + p_1 s_7)c^4 +  p_0 s_7 c^3 + p_1^2 c^2 + p_0^2)\\
        =&0.
    \end{aligned}
    $$
    From the coefficients of $c^2, c^4, c^7$ and $c^8$, we know that
    $$p_0=p_1=p_2=p_4=p_5=0.$$
    Further, we can get
    $$
    \begin{cases}
        p_6 s_1+t_1=0,\\
        p_6 s_2 + p_3 s_1 +t_2=0,
    \end{cases}
    $$
    with $s_1=1$, $s_2=a^3 + a b^2+ b^3$, $t_1=a^{12} b^3 +a^9 b^6+ a^6 b^9+ a^3 b^{12} $.
    Thus, we get
    $$
    \beta=a^3b^3c^3(a + b)^3(a^2 + a b + b^2)^3(a^2 b + a b^2 + c^3).
    $$
    By Lemma \ref{lemma-3}, we can know that $h=0$ has only one solution, which means that $H$ is a permutation.
\end{proof}

\begin{Th}
    \label{th-5}
    Let $m$ be odd and $F(x,y,z)=(f(x,y,z),f(y,z,x),f(z,x,y))$ be a function from $\gf_{2^m}^3$ to itself, where $$f(x,y,z)=x^3+x^2y+xy^2+x^2z+yz^2.$$
    Then $F$ is a permutation of $\gf_{2^m}^3$.
\end{Th}
\begin{proof}
    From Lemma \ref{lemma-5}(ii), it suffices to show that for all $(a,b,c) \in \mathbb{F}_{2^m}^3\backslash\{(0,0,0)\}$, the equation
    \begin{equation}
    \label{1.1}
        F(x+a,y+b,z+c)+F(x,y,z)=(0,0,0)
    \end{equation}
    has no solution in $\mathbb{F}_{2^m}^3$. It is trivial that Eq. \eqref{1.1} is equivalent to the equation system
    \begin{subequations}
    \renewcommand\theequation{\theparentequation.\arabic{equation}}
	      \label{v-eq1}
	   \begin{empheq}[left={\empheqlbrace\,}]{align}
        &(a+b+c)x^2  + (a + b)^2x + ay^2 + (a+ c)^2y + bz^2 + a^2z=f(a,b,c) ,\label{1.2.a}\\
        &cx^2 + b^2x + (a+b+c)y^2+ (b + c)^2y + bz^2 + (a +b)^2z=f(b,c,a) ,\label{1.2.b}\\
        &cx^2 + (b+ c)^2x + ay^2 + c^2y + (a+b+c)z^2+ (a + c)^2z=f(c,a,b) \label{1.2.c}.
        \end{empheq}
    \end{subequations}
    Adding the left side of the equation system \eqref{1.2.a},  \eqref{1.2.b} and  \eqref{1.2.c} together, we get
    \begin{equation*}
        (a+b+c)(x+y+z)(x+y+z+a+b+c)=(a+b+c)^3.
    \end{equation*}

    \textbf{Case 1}. When $a+b+c\neq 0$, the equation above is equivalent to
    $$(x+y+z)^2+(a+b+c)(x+y+z)+(a+b+c)^2=0,$$
    which has no solution in $\mathbb{F}_{2^m}^3$.

    \textbf{Case 2}. When $a+b+c=0$, Eq. \eqref{1.2.b} and Eq. \eqref{1.2.c} can be simplified as
    \begin{subequations}
    \renewcommand\theequation{\theparentequation.\arabic{equation}}
	      \label{v-eq2}
	   \begin{empheq}[left={\empheqlbrace\,}]{align}
        &(a+b)x^2 + b^2x + a^2y + bz^2 + (a+b)^2z=a^3+b^3 ,\label{1.3.a}\\
        &(a+b)x^2 + a^2x + ay^2 + (a+b)^2y +  b^2z=(a+b)^3+b^3 \label{1.3.b}.
        \end{empheq}
    \end{subequations}
    Let $Q_1=(a+b)x^2 + b^2x + a^2y + bz^2 + (a+b)^2z+a^3+b^3$ and $Q_2=(a+b)x^2 + a^2x + ay^2 + (a+b)^2y +  b^2z+(a+b)^3+b^3$, which can be seen as polynomials in $\gf_2[x,y,z,a,b,c]$.
    Then we compute the resultant of $Q_1, Q_2$ to eliminate the variable $x$ by MAGMA, that is
    \begin{equation}
    	\label{eq16}
        \begin{aligned}
            R_x(Q_1,Q_2)=(a+b)^2&[a^2y^4+b^2(a^2+ab+b^2)y^2+(a+b)(a^2+ab+b^2)^2y\\
            &+b^2z^4+a^2(a^2+ab+b^2)z^2+(a+b)(a^2+ab+b^2)^2z\\
            &+(a^2+ab+b^2)^3].
        \end{aligned}
    \end{equation}

    Firstly, if $a+b=0$, plug it into the summation of Eq. \eqref{1.3.a} and Eq. \eqref{1.3.b}, we have
    \begin{equation*}
       a^2y+az^2+ay^2+a^2z+a^3 = 0.
    \end{equation*}
    Let $Y=a^{-1}y$ and $\ Z=a^{-1}z$ since $a \neq0$ (otherwise we have $a=b=c=0$). The equation above is equivalent to
    $$a^3((Y+Z)^2+(Y+Z)+1)\neq0.$$
    Therefore, the equation system \eqref{v-eq1} has no solution in $\gf_{2^m}^3$.

    Secondly, if $a+b\neq 0$, we consider the condition that
    $a=0$, then $b\neq 0$.
    Let $Y=b^{-1}y$ and $Z=b^{-1}z$.  Eq. \eqref{eq16} is equivalent to
    $$R_x(Q_1,Q_2)=b^6((Y+Z^2+Z)^2+(Y+Z^2+Z)+1)\neq 0,$$
    which means that the equation system \eqref{v-eq1} has no solution in $\gf_{2^m}^3$.

    Next, we consider that $a\neq 0$.  Let $Y=a^{-1}y$, $Z=a^{-1}z$ and  $t= a^{-1}b$. Dividing Eq. \eqref{eq16} by $a^6(a+b)^2$, we can obtain that
    \begin{equation*}
        \begin{aligned}
         \frac{R_x(Q_1,Q_2)}{a^6(a+b)^2}=&Y^4+t^2(1+t+t^2)Y^2+(1+t)(1+t+t^2)^2Y\\
        &+t^2Z^4+(1+t+t^2)Z^2+(1+t)(1+t+t^2)^2Z\\
        &+(1+t+t^2)^3\triangleq D(Y,Z),
    \end{aligned}
    \end{equation*}
   where $D(Y,Z)$ is a function of $Y,Z$ on $\gf_{2^m}$. Then Eq. \eqref{eq16} has no solution if and only if $D(Y, Z)\neq 0$ for all $Y,Z\in \gf_{2^m}$, i.e.,
    \begin{equation}
    \label{1.5}
       \#\{(Y,Z)\in \gf_{2^m}^2~|~D(Y,Z)=0 \}=0.
    \end{equation}
    Based on the basic knowledge of exponential sums, we have the formula below
    \begin{equation*}
    \begin{aligned}
        &\#\{(Y,Z)\in \gf_{2^m}^2~|~D(Y,Z)=0 \}\\
        &=\frac{1}{2^m}\sum\limits_{Y,Z \in \mathbb{F}_{2^m}}\sum\limits_{\omega \in \mathbb{F}_{2^m}}(-1)^{\mathrm{Tr}_m(D(Y,Z)\omega)}\\
        &=\frac{1}{2^m}\sum\limits_{\omega \in \mathbb{F}_{2^m}}(-1)^{\mathrm{Tr}_m((1+t+t^2)^3\omega)}
            \sum\limits_{Y \in \mathbb{F}_{2^m}}(-1)^{\mathrm{Tr}_m(M_1(\omega)Y^4)}\quad\ \sum\limits_{Z \in \mathbb{F}_{2^m}}(-1)^{\mathrm{Tr}_m(M_2(\omega)Z^4)}\\
        &\triangleq \frac{1}{2^m}S,
    \end{aligned}
    \end{equation*}
    where $M_1(\omega)=\omega+\omega^2t^4(1+t+t^2)^2+\omega^4(1+t)^4(1+t+t^2)^8$, $M_2(\omega)=\omega t^2+\omega^2(1+t+t^2)^2+\omega^4(1+t)^4(1+t+t^2)^8$. Note that the second equality sign holds due to the properties of the trace function.

    Further, by Lemma \ref{lemma-1}, we have
    $$
    S=2^{2m} \sum\limits_{\omega \in \mathbb{F}_{2^m},M_1(\omega)=M_2(\omega)=0}(-1)^{\mathrm{Tr}_m((1+t+t^2)^3\omega)}.
    $$
    Let $I=\{ \omega\in \gf_{2^m}~|~M_1(\omega)=M_2(\omega )=0\}$. When $\omega \in I$, it satisfies that
    \begin{equation}
    \label{1.7}
        \begin{cases}       \omega+\omega^2t^4(1+t+t^2)^2+\omega^4(1+t)^4(1+t+t^2)^8=0,\\
        \omega t^2+\omega^2(1+t+t^2)^2+\omega^4(1+t)^4(1+t+t^2)^8=0.
        \end{cases}
    \end{equation}
Computing the sum of the two equations in Eq. \eqref{1.7}, we can get
    $$
    (1+t)^2\omega+(1+t)^4(1+t+t^2)^2\omega^2=0.
    $$
    Since $a+b\neq0$ and $1+t+t^2\neq 0$ , we know that
    $\omega=0$ and  $\omega=\frac{1}{(1+t)^2(1+t+t^2)^2}$ are the solutions. Moreover, we can easily verify that $$I= \left\{0, \frac{1}{(1+t)^2(1+t+t^2)^2}\right\}.$$

    Therefore, we have
    \begin{equation*}
        \begin{aligned}
        S&=2^{2m}\sum\limits_{\omega \in I}(-1)^{\mathrm{Tr}_m\left((1+t+t^2)^3\omega\right)}\\
        &=2^{2m}\left[(-1)^0+(-1)^{\mathrm{Tr}_m\left(\frac{1+t+t^2}{1+t^2}\right)}\right]\\
        &=2^{2m}\left[(-1)^0+(-1)^{\mathrm{Tr}_m\left(1+\frac{1}{1+t}+\left(\frac{1}{1+t}\right)^2\right)}\right]\\
        &=2^{2m}\left[(-1)^0+(-1)^1\right]=0.
        \end{aligned}
    \end{equation*}

    Thus Eq. \eqref{1.5} holds and Eq. \eqref{eq16} has no solution, which means $F$ is a permutation.
\end{proof}

\begin{Th}
    \label{th1}
    Let $m$ be odd and $F(x,y,z)=(f(x,y,z),f(y,z,x),f(z,x,y))$, where  $$f(x,y,z)=x^3+yz^2+y^2z.$$ Then $F$ is a  permutation of $\gf_{2^m}^3$.
\end{Th}
\begin{proof}
      It suffices to show that the equation system
            \begin{subequations}            \renewcommand\theequation{\theparentequation.\arabic{equation}}
	      \label{i-eq1}
	   \begin{empheq}[left={\empheqlbrace\,}]{align}
		&x^3+yz^2+y^2z=a \label{i-num1} \\
        &y^3+x^2z+xz^2=b \label{i-num2} \\
        &z^3+xy^2+x^2y=c\label{i-num3}
	\end{empheq}
            \end{subequations}
has exactly one solution in $\gf_{2^m}^3$ for all $a,b,c\in\gf_{2^m}$.

Computing the summation of  the equation system (\ref{i-eq1}), we can get
        $$
        (x+y+z)^3=a+b+c.
        $$
         Let $d=(a+b+c)^{1/3}$. Then we have $x+y+z=d$. Plugging it into Eq. (\ref{i-num1}) and Eq. (\ref{i-num3}), we obtain
        $$
        \begin{cases}
            (x+y)^3+(y+d)^3=a+d^3,\\
            (x+d)^3+(y+d)^3=c+d^3.
        \end{cases}
        $$
        Let $X=x+d,\ Y=y+d$. Then the equation system above is equivalent to
        $$
        \begin{cases}
            (X+Y)^3+Y^3=a+d^3,\\
            X^3+Y^3=c+d^3.
        \end{cases}
        $$
        By simplifying, we get
            \begin{subequations}
            \renewcommand\theequation{\theparentequation.\arabic{equation}}
	      \label{i-eq2}
	   \begin{empheq}[left={\empheqlbrace\,}]{align}
            &X(X^2+XY+Y^2)=b+c,\label{i-num4}\\
            &(X+Y)(X^2+XY+Y^2)=a+b.\label{i-num5}
            \end{empheq}
            \end{subequations}

        Next, we investigate the solutions of the equation system \eqref{i-eq2} in the case of $b+c=0$ and $b+c\neq0$.

       {\textbf{Case 1.} }  If $b=c$, we get $X^2+XY+Y^2=X^2\left( 1+\frac{Y}{X} + \frac{Y^2}{X^2} \right)=0$ by Eq. \eqref{i-num4} when $X\neq0$, which implies $\frac{Y}{X}\in\gf_{2^2}$ and contradicts the fact that $m$ is odd. Thus, we have $X=0$. Plugging it into Eq. \eqref{i-num5}, we have $Y^3=a+b$. Furthermore, $(d,d+(a+b)^{1/3},d+(a+b)^{1/3})$
        is the unique solution of the equation system \eqref{i-eq1}.

       \textbf{Case 2.}  If $b\neq c$, then $X\neq 0$. After dividing the left part of Eq. (\ref{i-num5}) by that of Eq. (\ref{i-num4}) and simplifying, we get
        $$
        \frac{Y}{X}=\frac{a+c}{b+c}.
        $$
        Plugging it into Eq. (\ref{i-num5}) implies
        $$
        \left[1+\left(\frac{a+c}{b+c}\right)^3\right]X^3=a+b.
        $$
       Finally, the only solution of the equation system \eqref{i-eq1} can be solved. That is,
        $$
        \begin{cases}
            x =d+\left[\frac{(a+b)(b+c)^3}{(a+c)^3+(b+c)^3}\right]^{1/3},\\
            y =d+\left(\frac{a+c}{b+c}\right)\left[\frac{(a+b)(b+c)^3}{(a+c)^3+(b+c)^3}\right]^{1/3}, \\
            z =d+ \left(\frac{a+b}{b+c}\right)\left[\frac{(a+b)(b+c)^3}{(a+c)^3+(b+c)^3}\right]^{1/3}\ .
        \end{cases}
        $$

        To sum up, $F(x,y,z)$ is a   permutation of $\gf_{2^m}^3$.
\end{proof}

\begin{Th}
    \label{th2}
     Let $m$ be odd and $F(x,y,z)=(f(x,y,z),f(y,z,x),f(z,x,y))$, where  $$f(x,y,z)=x^3+y^3+x^2y+x^2z+yz^2.$$ Then $F$ is a  permutation of $\gf_{2^m}^3$.
\end{Th}
\begin{proof}
   It suffices to show that for all $a,b,c\in\gf_{2^m}$, the equation system
        \begin{subequations}  \renewcommand\theequation{\theparentequation.\arabic{equation}}
	      \label{ii-eq1}
	   \begin{empheq}[left={\empheqlbrace\,}]{align}
    &x^3+y^3+x^2y+x^2z+yz^2=a\label{ii-num1}\\
    &y^3+z^3+y^2z+y^2x+zx^2=b\label{ii-num2}\\
    &z^3+x^3+z^2x+z^2y+xy^2=c\label{ii-num3}
        \end{empheq}
        \end{subequations}
has exactly one solution in $\gf_{2^m}^3$.

Computing the summation of Eq. (\ref{ii-num1}) and Eq. (\ref{ii-num2}), and that of Eq. (\ref{ii-num2}) and Eq. (\ref{ii-num3}) resp. , we get
        \begin{equation}
            \begin{cases}
            \label{ii-eq2}
            (x+y)^3+(y+z)^3=a+b,\\
            (x+z)^3+(y+z)^3=b+c.
        \end{cases}
        \end{equation}
And computing the summation of Eq.  (\ref{ii-num1}), Eq. (\ref{ii-num2}) and Eq. (\ref{ii-num3}) , we get
        \begin{equation}
        \label{ii-eq3}
            x^2y+y^2z+z^2x=a+b+c.
        \end{equation}
Let $X=x+z$ and $Y=y+z$. The equation system (\ref{ii-eq2}) is equivalent to
        \begin{subequations}        \renewcommand\theequation{\theparentequation.\arabic{equation}}
	      \label{ii-eq4}
	   \begin{empheq}[left={\empheqlbrace\,}]{align}
            &(X+Y)^3+Y^3=a+b,\label{ii-num4}\\
            &X^3+Y^3=b+c.\label{ii-num5}
            \end{empheq}
        \end{subequations}
        Furthermore, after simplifying the equation system (\ref{ii-num4}) and (\ref{ii-num5}), we get
        \begin{equation}
        \label{eqeq1}
        \begin{cases}
            X(X^2+XY+Y^2)=a+b,\\
            (X+Y)(X^2+XY+Y^2)=b+c.
        \end{cases}
        \end{equation}

        \textbf{Case 1.} If $a=b$, then $X=0$ since $X^2+XY+Y^2=0$ if and only if  $X=Y=0$. SO we have $Y^3=b+c$ from Eq. (\ref{ii-num5}).
         Moreover, by Eq. (\ref{ii-eq2}), we get the unique solution that
        $$
        \begin{cases}
            x=b^{1/3}+(b+c)^{1/3},\\
            y=b^{1/3},\\
            z=b^{1/3}+(b+c)^{1/3}.
        \end{cases}
        $$

        \textbf{Case 2.} If $a\neq b$, then $X\neq 0$ and $X^2+XY+Y^2\neq 0$.

        The equation system \eqref{eqeq1} is equivalent to
        \begin{equation}
        \label{ii-eq4}
            Y=\frac{a+c}{a+b}X.
        \end{equation}
       Plugging Eq. (\ref{ii-eq4}) into Eq. (\ref{ii-num5}), we have
       $$
       \left[1+\left(\frac{a+c}{a+b}\right)^3\right]X^3=b+c.
       $$

 If $b+c=0$, we know that $X=Y$ (i.e., $x=y$) from Eq. (\ref{ii-eq4}). And Eq. (\ref{ii-num4}) and Eq. (\ref{ii-num3}) are equivalent to
        $$
        \begin{cases}
            X^3=Y^3=a+b,\\
            z=c^{1/3},
        \end{cases}
        $$
        respectively.   Further, we can solve the equation system \eqref{ii-eq1} and obtain the following solutions.
        $$
        \begin{cases}
          x=(a+b)^{1/3}+c^{1/3},\\
          y=(a+b)^{1/3}+c^{1/3},\\
          z=c^{1/3}.
        \end{cases}
        $$

 If $b+c\neq 0$, we can straightly get the solutions of $X$ and $Y$,
        $$
        \begin{cases}
            X=\frac{(a+b)(b+c)^{1/3}}{[(a+b)^3+(a+c)^3]^{1/3}}\ , \\
            Y=\frac{a+c}{a+b} X\ .
        \end{cases}
        $$

        It is apparent that $x+y= X+Y$. The left side of Eq. (\ref{ii-eq3}) is equivalent to
        $$
        \begin{aligned}
           &x^2(X+Y+x)+xz^2+(X+Y+x)^2z\\
           &=(x+z)^3+(X+Y)x^2+z^3+(X+Y)^2z\\
           &=X^3+(z+X+Y)^3+(X+Y)^3+(X+Y)(x+z)^2\\
            &=a+b+c.
        \end{aligned}
        $$
        That is,
        $$
        (z+X+Y)^3+(X+Y)^3+X^2Y=a+b+c.
        $$
        Thus the value of $z$ can be easily solved, and so do $x,y$. The solutions are as follows.
        $$
        \begin{cases}
            x=\left[\frac{b(a+c)^3+(a+b)^2(c^2+ab)}{(a+c)^3+(b+c)^3}\right]^{1/3}+Y\ , \\
            y=\left[\frac{b(a+c)^3+(a+b)^2(c^2+ab)}{(a+c)^3+(b+c)^3}\right]^{1/3}+X\ ,\\
            z=\left[\frac{b(a+c)^3+(a+b)^2(c^2+ab)}{(a+c)^3+(b+c)^3}\right]^{1/3}+X+Y\ .
        \end{cases}
        $$
        To sum up, the equation system \eqref{ii-eq1} has only one solution.
\end{proof}

\begin{Th}
    \label{th3}
    Let  $m$ be odd and $F(x,y,z)=(f(x,y,z),f(y,z,x),f(z,x,y))$ be a function from $\gf_q^3$ to itself,  where  $$f(x,y,z)=x^3+x^2y+xy^2+x^2z+xz^2.$$ Then $F$ is a permutation of $\gf_{2^m}^3$.
\end{Th}

\begin{proof}
        It suffices to show that the equation system
        \begin{subequations}
        \renewcommand\theequation{\theparentequation.\arabic{equation}}
	      \label{iii-eq1}
	   \begin{empheq}[left={\empheqlbrace\,}]{align}
        &x^{3}+x^2y+xy^2+x^2z+xz^2=a
        \label{iii-num1}\\
        &y^{3}+y^2z+yz^2+y^2x+yx^2=b
        \label{iii-num2}\\
        &z^{3}+z^2x+zx^2+z^2y+zy^2=c
        \label{iii-num3}
        \end{empheq}
        \end{subequations}
        has a unique solution in $\gf_{2^m}^3$ for all $a,b,c\in\gf_{2^m}$. Computing the summation of the equation system (\ref{iii-num1}), (\ref{iii-num2}) and (\ref{iii-num3}), we get
        \begin{equation}
        \label{iii-eq2}
            x^3+y^3+z^3=a+b+c.
        \end{equation}

     And computing the left part of Eq. (\ref{iii-num1}) and (\ref{iii-num2}), and that of Eq. (\ref{iii-num2}) and (\ref{iii-num3}) resp. , we obtain
        $$
        \begin{cases}
            (x+z)^{3}+(y+z)^{3}=a+b,\\
            (x+y)^{3}+(y+z)^{3}=a+c.
        \end{cases}
        $$
        Let $X=x+z$, $Y=y+z$. The equation system above is equivalent to
        \begin{subequations}  \renewcommand\theequation{\theparentequation.\arabic{equation}}  \label{iii-eq3}
	   \begin{empheq}[left={\empheqlbrace\,}]{align}
    &(X+Y)^{3}+Y^{3}=X(X^2+XY+Y^2)=a+c, \label{iii-num4}\\
    &X^{3}+Y^{3}=(X+Y)(X^2+XY+Y^2)=a+b.
        \label{iii-num5}
        \end{empheq}
        \end{subequations}

        Similarly, we investigate the solutions of the equation system\eqref{iii-eq3} when $a+c=0$ or not.

        \textbf{Case 1.} If $a=c$, we have $X=0$, i.e., $x=z$ and $Y^3=a+b$. Plugging them into Eq. (\ref{iii-eq2}), the values of $x,y,z$ can be obtained as follows.
        $$
        \begin{cases}
            x=b^{1/3}+(a+b)^{1/3},\\
            y=b^{1/3},\\
            z=b^{1/3}+(a+b)^{1/3}.
        \end{cases}
        $$

      \textbf{Case 2.} If $a\neq c$, 
 we have $X\neq 0$. The method we use here is exactly the same as the one used in case 2 of Theorem \ref{th2}. We omit it here and post the final solution directly.
     \begin{itemize}
    \item  When $a=b$, the solution of the equation system  \eqref{iii-eq1} is
          $$
            \begin{cases}
            x=(a+c)^{1/3}+c^{1/3},\\
            y=(a+c)^{1/3}+c^{1/3},\\
            z=c^{1/3}.
            \end{cases}
          $$
          \item When $a\neq b$, the solution is
          $$
      \begin{cases}
        x=A+z,\\
        y=B+z,\\
        z=\left[\frac{a^4+b^4}{(a+c)^3+(b+c)^3}+c\right]^{1/3}+\frac{(a+b)^{4/3}}{[(a+c)^3+(b+c)^3]^{1/3}}.
      \end{cases}
      $$
      where
      $$
      A=\frac{(a+c)(a+b)^{1/3}}{[(a+c)^3+(b+c)^3]^{1/3}}, B=\frac{b+c}{a+c} \frac{(a+c)(a+b)^{1/3}}{[(a+c)^3+(b+c)^3]^{1/3}}.
      $$
      \end{itemize}

\end{proof}

In the final of this section, we discuss the QM-equivalent relation between the newly-constructed permutations and the known ones.
Clearly, given a permutation of $\gf_{2^m}^3$, we can obtain a permutation polynomial of $\gf_{2^{3m}}$. Let $\omega\in \gf_{2^{3m}}$ satisfy that $\{1,\omega,\omega^2\}$ forms a basis of $\gf_{2^{3m}}$ over $\gf_{2^m}$. Then any $t\in \gf_{2^{3m}}$ can be expressed by $x,y,z \in \gf_{2^m}$ as $t=x+y\omega+z\omega^2$. Raising the equation to its $2^m$\textit{-}th and $2^{2m}$\textit{-}th power,  we get an equation system. Based on the basic knowledge of linear algebra, $x,y,z$ can be expressed by $t, t^q, t^{q^2}$, denoted as $x(t),y(t),z(t)$. Therefore, given a permutation $F(x,y,z)=(f_1(x,y,z),f_2(x,y,z),f_3(x,y,z))$, the corresponding permutation polynomial is $F'(t)=f_1(x(t),y(t),z(t))+f_2(x(t),y(t),z(t))\omega+f_3(x(t),y(t),z(t))\omega^2$.

In the literature,  many researchers constructed permutation trinomials over $\gf_{q^3}$. Results of the recent survey are presented in Table \ref{label-2}.
\begin{table}[h]
    \caption{Known Classes of Permutation Polynomials on $\gf_{2^{3m}}$}
    \label{label-2}
    \centering
    \begin{tabular}{cccccc}
    \toprule
    $i$ & $f_i(x)\in \gf_{2^{3m}}[x]$ & Ref.
    &$i$ & $f_i(x)\in \gf_{2^{3m}}[x]$ & Ref.\\
    \hline
    $1$ & $v^{-1}x^{q^2+q+2}$ &\cite{tu2014several} & $8$ & $ax+L(x^s),L(x)=x+bx^q$&\cite{zhenglijing2023two}\\
    $2$ & $vx+x^{q+1}+x^{q^2+1}$ &\cite{tu2014several} &$9$ & $ax+L(x^s),L(x)=x+bx^{q^2}$&\cite{zhenglijing2023two}\\
    $3$ & $cx+\tr_{q^{2n+1}/q}(x^a)$ & \cite{li2018permutation}&  $10$ &$x+x^{q^2-q+1}+x^{q^2+q-1}$ &\cite{wang2018six} \\
    $4$ & $x^{q+1}+(x+x^q)^{2q^2}$ &\cite{gong2016permutation} &
    $11$ & $x+x^{q^2}+x^{q^2+q-1}$&\cite{wang2018six}\\
    $5$ &  $x^d+L(x^s),L(x)=ax+bx^q$ &\cite{pang2021permutation}&
    $12$ & $x+x^{q^2+q-1}+x^{q^3-q^2+q}$ & \cite{li2018permutation}\\
    $6$ & $x^d+L(x^s),L(x)=x+Ax^q$ &\cite{gupta2022new}&$13$ &$x+x^{q^2}+x^{q^3-q^2+1}$ & \cite{li2018permutation}
     \\
    $7$ & $x^d+L(x^s),L(x)=x+Ax^{q^2}$ &\cite{gupta2022new}&
    $14$ &$x^{q^2-q+1} + Ax^{q^2}+ bx$ & \cite{bartoli2023permutation}\\

    \bottomrule
    \end{tabular}
\end{table}

It is trivial that the corresponding permutation polynomials of $\gf_{2^{3m}}$ of the newly constructed permutations in this section are QM\textit{-}inequivalent to the known permutation trinomials since ours have more terms.  In fact, in most cases, the simpler structure of permutation on $\gf_{2^m}^3$ has, the more complex its corresponding permutation is.  In the following, we will give an explicit example.

\begin{example}
     For the permutation $F(x,y,z)=(x^3+yz^2+y^2z,y^3+zx^2+z^2x,z^3+xy^2+x^2y)$ on $\gf_{2^3}^3$ , its corresponding PP is $F'(x)=\omega^{31}x^{12} + \omega^{213}x^{10} + \omega^{104}x^9 + \omega^{320}x^8 + \omega^{487}x^6 + \omega^{240}x^5 + \omega^{238}x^4 + \omega^{435}x^3$. By Remark \ref{rem}, $F'(x)$ is QM-inequivalent to the known permutation trinomials over $\gf_{2^9}$.
\end{example}

\section{Conclusion and further works}
\label{ch4}
 In this paper, we mainly focused on constructing several new classes of $3$-homogenous rotatable permutations from $\gf_q^3$ to itself. In our proofs, the resultant of polynomials and some skills of exponential sums were used.
 Clearly, many directions could be further investigated, and some perspectives are listed here.
 First, from the experimental results by MAGMA, we can find that there are still some $3$-homogenous rotatable permutation instances on $\gf_{2^m}^3$ where $m=3,5,7$. It is interesting to generalize them into infinite classes. Second,
 due to the diversity of $d$-homogeneous polynomials, there surely are plenty of possibilities to enrich PPs of the form $f(x)=x^r h(x^{q-1})$ over $\gf_{q^n}$ using the method of \cite{qu2023}. Therefore,  constructing $d$-homogeneous permutations for more structures and higher powers deserves to be explored.
Finally, from the point of properties of permutations, the cryptographic properties of these construction results are unexplored in this paper. We believe that,  if permutations with good properties can be discovered from them, they will be able to broaden the applications of multivariate permutations.

\bibliographystyle{plain}
\bibliography{ref}

\end{document}